\documentclass[11pt]{article}

\usepackage{pdfsync}

\usepackage{a4wide}
\usepackage{amsmath,amssymb}
\usepackage{amsfonts}
\usepackage[utf8]{inputenc}
\usepackage[english]{babel}
\usepackage{color}
\usepackage{hyperref}
\usepackage{enumerate}
\usepackage{graphicx}
\usepackage[all]{xy}
\usepackage{amsthm}
\usepackage{xkeyval,calc,listings, fp}
\usepackage{CJKutf8}
\usepackage{bm}
\usepackage{wrapfig}
\usepackage{caption}
\usepackage[normalem]{ulem}
\usepackage[colorinlistoftodos]{todonotes}
\usepackage{cleveref}
\usepackage[round, authoryear]{natbib}

\newtheorem{lemma}{Lemma}

\newtheorem{theorem}{Theorem}
\newtheorem{corollary}{Corollary}

\newtheorem{examp}{Example}
\newtheorem{remark}{Remark}
\newtheorem*{mainthm}{Main Theorem}

\newcommand{\produ}[1]{\langle #1\rangle}

\newcommand{\reals}{\mathbb{R}}

\newcommand{\cof}[1]{s_{#1}}
\newcommand{\bef}[2]{\beta^{#1}_{#2}}
\newcommand{\tF}{\mathcal{F}}
\newcommand{\Ric}{\mathrm{Ric}}
\newcommand{\CD}{\mathrm{CD}}
\newcommand{\A}{\mathrm{A}}
\newcommand{\U}{\mathbf{U}}
\newcommand{\J}{\mathbf{J}}
\newcommand{\e}{\mathrm{e}}
\newcommand{\V}{\mathbf{V}}
\newcommand{\D}{\mathrm{D}}

\title{Rigidity of the Borell-Brascamp-Lieb Inequality on Weighted Riemannian Manifolds}
\author{Rongkai Zhang\thanks{e-mail: u284872k@ecs.osaka-u.ac.jp} \\Graduate School of Science, the University of Osaka, Osaka 560-0043, Japan }

\begin{document}
\maketitle

\begin{abstract}
In this paper we discuss some results regarding the rigidity of the Borell-Brascamp-Lieb inequality and the Brunn-Minkowski inequality. We show a theorem of rigidity on curvature and measure of the Borell-Brascamp-Lieb inequality, a generalisation of the curvature rigidity theorem by Balogh-Krist\'aly in \citep{BaloghZoltanM.2018EiBi} to the weighted setting. We present some rigidity results of the Brunn-Minkowski inequality and a few further open problems.
\end{abstract}

\section{Introduction}

Since the seminal works on the application of optimal transport in geometric analysis including \citep{OttoF.2000GoaI}, \citep{STURMKarl-Theodor2006Otgo} and \citep{LottJohn2009Rcfm}, it became possible to study a wide range of analytical problems on Riemannian manifolds using the tool of optimal transport. In particular, the theory of optimal transport has been a powerful tool in the study of functional and geometrical inequalities on Riemannian manifolds with curvature and dimension bounds.

First proved in 1952 by Henstock and MacBeth on $\reals$ by \citet{HenstockR.1953OtMo} and later developed by \citet{BorellC.1975Csfi} and \citet{BrascampHermJan1976Oeot}, the Borell-Brascamp-Lieb inequality is an integral inequality concerned with the $L^1$-norm of two functions and that of the mean of them.

On the $n$-dimensional Euclidean space $\reals^n$, the Borell-Brascamp-Lieb inequality states that if functions $f,g,h >0$ such that for $s \in (0,1)$, $$h((1-s)x + sy) \geq \mathcal M^p_s(f(x),g(y)) \quad \forall (x,y) \in \reals^n \times \reals^n,$$
then $$\|h\|_{L^1} \geq \mathcal M^{\frac{p}{1+np}}_s(\|f\|_{L^1},\|g\|_{L^1})$$
where for $p \in [-\frac{1}{n},\infty] $ and real numbers $a,b \geq 0$,
 \begin{eqnarray*}
        \mathcal M^p_s(a,b):= \left\{ \begin{array}{cl}
        ((1-s)a^p + sb^p)^\frac{1}{p} \quad & \mbox{when $ab \neq 0, $}\\
        0 \quad & \mbox{when $ab=0$}\\
    \end{array}
    \right.
    \end{eqnarray*}
is the $p$-mean, with the conventions $\mathcal M^{\infty}_s(a,b)= \max\{a,b\}$, $\mathcal{M}^0_s(a,b) = a^{1-s}b^s$ and $\mathcal{M}^{-\infty}_s=\min\{a,b\}$.

The Borell-Brascamp-Lieb inequality was proved on Riemannian manifolds by Cordero-Erausquin, McCann and Schmuckenschl\"ager in \citep{Cordero-ErausquinDario2001ARii}. It was then generalised further to metric measure spaces by \citet{BacherKathrin2010OBIo}. The Borell-Brascamp-Lieb inequality implies a number of functional and geometric inequalities including the Brunn-Minkowski inequality and the Pr\'ekopa-Leindler inequality, a diagram showing their relationship can be seen in \citep{GardnerR.J.2002TBi}.

Rigidity of the Borell-Brascamp-Lieb inequality on $\reals^n$ has been studied since 1970s, various papers including \citep{DubucSerge1977Cdce} since then have discussed this problem. Rigidity of it on (unweighted) Riemannian manifolds with curvature bounded from below was shown by Balogh-Krist\'aly in \citep{BaloghZoltanM.2018EiBi}. 

This paper generalises the curvature rigidity result of the Borell-Brascamp-Lieb inequality on Riemannian manifolds by Balogh and Krist\'aly, Theorem 4.1 of \citep{BaloghZoltanM.2018EiBi}, to weighted Riemannian manifolds. The main result of this paper states that the rigidity of the Borell-Brascamp-Lieb inequality implies identical sectional curvatures and certain behaviour of the measure.

\begin{mainthm}[Theorem \ref{rigthmgen}]
    On a weighted Riemannian manifold $(M,g,m = \e^{-\psi}\mathrm{vol}_g)$, satisfying $\Ric_{m,N} \geq Kg$ where $n \leq N < \infty$ and $K \in \reals$, equality of the Borell-Brascamp-Lieb inequality implies that the $N$-Ricci curvature along the geodesic of transport $\gamma(t):= \tF_t(x)$ equals $K$ and the sectional curvatures along the geodesic $\gamma$ are identically $\frac{K}{N-1}$. The measure behaves as an $(N-n)$-th power of a linear combination of trignometric functions along the geodesic $\gamma(t)$ when $K>0$; as an $(N-n)$-degree polynomial of $t$ along $\gamma(t)$ when $K=0$; and as an $(N-n)$-th power of of a linear combination of hyperbolic functions along $\gamma(t)$ when $K<0$.
\end{mainthm}

The article is structured as follows. In section 2, we first demonstrate a proof of an interpolation inequality on weighted Riemannian manifolds, followed by a proof of the Borell-Brascamp-Lieb inequality on weighted Riemannian manifolds. In section 3, we discuss the rigidity of the Borell-Brascamp-Lieb inequality and prove the main theorem, Theorem \ref{rigthmgen}, and in section 4 we discuss the Brunn-Minkowski inequality and its rigidity, as a consequence of the rigidity of the Borell-Brascamp-Lieb inequality. At the end in section 5 we present a few related open questions.

\section*{Acknowledgements}
I would like to express my deep gratitude to Prof. Shin-ichi Ohta for his guidance, suggestions and discussions. I would also like to thank Dr. Hiroshi Tsuji and Mr. Zisu Zhao for some helpful discussions.

\noindent This work was supported by JST SPRING, Grant Number JPMJSP2138.

\section{Preliminaries}
Throughout the paper, we consider a smooth weighted Riemannian manifold $(M,g,m)$ of dimension $n$, where $m=\e^{-\psi}\mathrm{vol}_g$, satisfying the condition $\Ric_{m,N}(\cdot) \geq Kg(\cdot, \cdot)$, $N \in [n,\infty)$ and $K \in \reals$. Recall that the $N$-Ricci curvature $\Ric_{m,N}$ is defined to be: 
\begin{eqnarray*}
    \Ric_{m,N}(v) := \left\{ \begin{array}{cl}
        \Ric_m(v) - \frac{(\partial_v \psi)^2}{N-n} \quad & \mbox{for $N \in (n,\infty)$},\\
        \Ric_m(v) \quad & \mbox{for $N =\infty$},\\
        \Ric(v) \quad & \mbox{for $N=n$, $\psi = \mathrm{const.}$}
    \end{array}
    \right.
\end{eqnarray*}
where $\Ric_m(v) = \Ric(v) + \mathrm{Hess}_\psi(v,v) $.

We denote $\mathcal{P}(M)$ to be the space of Borel probability measures on $M$, and let $\mathcal{P}^2(M)$ be the space of measures with finite second moment
$$\mathcal{P}^2(M) = \Big\{ \mu \in \mathcal{P}(M) \Big| \int_M d(x,x_0)^2 d\mu < \infty  \text{ for arbitrary } x_0 \in M \Big\}.$$

By the theorem of Brenier-McCann, given two measures on $(M,g,m)$, $\mu_0=\rho_0 m = \rho_0\e^{-\psi}\mathrm{vol}_g$ and $\mu_1=\rho_1 m = \rho_1\e^{-\psi}\mathrm{vol}_g$ in $\mathcal{P}^2(M)$ that are absolutely continuous with respect to $m$, there is a unique optimal transference plan from $\mu_0$ to $\mu_1$, given by $\tF_t$, $\mu_t = (\tF_t)_{\#}\mu_0$, the push-forward of $\mu_0$ by $\tF_t(x):=\exp_x(t\nabla f(x))$, and $f:M \rightarrow \reals$ is a $\frac{d^2}{2}$-convex function. We refer to \citep{alma995800334401591} for the fundamentals of optimal transport theory, see also \citep{McCannR.J.2001Pfom} and \citep{FigalliAlessio2011LsoK} for the theory of optimal transport on Riemannian manifolds.

We consider the \textit{weighted Jacobian} $J^\psi_t(x) := \e^{\psi(x) - \psi(\tF_t(x))} \|(D\tF_t(x))\|$ along the geodesic $\gamma(t):= \tF_t(x)$. We have the following concavity condition on the weighted Jacobian: 
\begin{equation}\label{eqn2}
    J^\psi_t(x)^{\frac{1}{N}} \geq (1-t) \bef{1-t}{K,N}(d(x, \tF_1(x)))^{\frac{1}{N}} + t \bef{t}{K,N}(d(x, \tF_1(x)))^{\frac{1}{N}}J^\psi_1(x)^{\frac{1}{N}}
\end{equation}
for $\mu_0$-a.e. $x$ such that $\tF_t$ is differentiable at $x$ and all $t \in [0,1]$, with the function $\bef{t}{K,N}$ being defined as:
$$ \bef{t}{K,N}(r) = \Big( \frac{\cof{K,N}(tr)}{t\cof{K,N}(r)} \Big)^{N-1} $$
where 
\begin{eqnarray*}
    \cof{K,N}(r) := \left\{ \begin{array}{cl}
        \sqrt{\frac{N-1}{K}} \sin \Big(r \sqrt{\frac{K}{N-1}}\Big) \quad & \mbox{if $K >0$},\\
        r \quad & \mbox{if $K =0$},\\
        \sqrt{-\frac{N-1}{K}} \sinh \Big( r \sqrt{\frac{-K}{N-1}} \Big) \quad & \mbox{if $K<0$}.
    \end{array}
    \right.
\end{eqnarray*}
It is known that the density function $\rho_t$ given by $\mu_t = \rho_t m$ satisfies the Monge-Amp\`ere equation
\begin{equation} \label{MAE}
    \rho_0(x) = \rho_t(\tF_t(x))J^\psi_t(x) \text{ for } \mu_0\text{-a.e. } x.
\end{equation} 

The interpolation inequality \eqref{eqn2} on unweighted Riemannian manifolds was shown in the proof of Theorem 1.7 of \citep{STURMKarl-Theodor2006Otgo}, and its Finsler version was shown in the proof of Theorem 1.2 of \citep{OhtaShin-ichi2009Fii}. Here we present a proof on weighted Riemannian manifolds satisfying $\Ric_{m,N} \geq Kg$, along the computation shown in the proof of Theorem 1.7 of \citep{STURMKarl-Theodor2006Otgo}, with an addition of the factor due to measure to the proof in \citep{STURMKarl-Theodor2006Otgo}, for later use. See also pages 393-401 of \citep{alma999885620001591}.
\begin{lemma} \label{cvxwtj}
    On a weighted Riemannian manifold $(M,g,m = e^{-\psi(x)}\mathrm{vol}_g)$ satisfying $\Ric_{m,N} \geq Kg$, the weighted Jacobian $J^\psi_t(x) := e^{\psi(x) - \psi(\tF_t(x))} \|(D\tF_t(x))\|$ defined along the geodesic $\gamma(t):= \tF_t(x)$, where $\tF_t(x)=\exp_x(t\nabla f(x))$ described as above satisfies the concavity condition \eqref{eqn2}.
\end{lemma}
\begin{proof}
We fix $x \in \mathrm{supp} \mu_0$ on $M$, and consider the geodesic $$\gamma(t): [0,1] \rightarrow M, t \mapsto \tF_{t}(x)$$ representing the path of transport starting from $x$. 

When $\gamma(t)$ is constant, inequality \eqref{eqn2} is trivial as $J_t^\psi(x) = 1$ identically.

When $\gamma(t)$ is nonconstant, at $x$, we consider the tangent space $T_xM = T_{\tF_0(x)}M$ and an orthonormal frame $\{ e_0^i \}_{i=1}^n$, where we denote $e^1_0$ to be the unit tangent of the geodesic: $\frac{\gamma'(0)}{\|\gamma'(0)\|}$. We can thus define $n-1$ Jacobi fields $\J_i(t) = \frac{\partial}{\partial s}\exp_x(t(\nabla f(x) + se^i_0))$ for $i \in \{2,...,n\}$ that are orthogonal to $\gamma'(t)$, and the vectors $\{\J_i(t)\}$ form a basis of $\{\gamma'(t)\}^\perp \subset T_{\tF_t(x)}M$ along the entire geodesic $\gamma$. 

We denote $\A_t(x)$ to be the differential map:
$$\A_t(x) := \D\tF_t(x): T_xM \rightarrow T_{\tF_t(x)}M.$$
It is known that $\A_t$ satisfies the Riccati equation $$\D_t\D_t\A_t(x) + R(\A_t(x),\gamma'(t))\gamma'(t)=0$$ with the initial condition $A_0=\mathrm{Id}$ and $\D_tA_t|_{t=0} = -\mathrm{Hess} \psi$, see page 142 of \citep{STURMKarl-Theodor2006Otgo}. 

Since $\{J_i\}_{i \geq 2}$ form a basis of the vector space $\{\gamma'(t)\}^\perp$, and since for $i \geq 2$, $\produ{\J_i'(t),\gamma'}=0$ by Gau{\ss} lemma, we can define, as in \citep{STURMKarl-Theodor2006Otgo}, the matrix $\U(t)$ to be such that $\U(t) \A_t := \D_t\A_t$. And the Riccati equation can also be written as
$$\U' + \U^2 + R\A_t^{-1}=0 \text{, or } \U' + \U^2 + R( \cdot, \partial_t\tF_t)\partial_t\tF_t =0.$$

At $\gamma(t)$ we set an orthonormal basis $\{e^i_t\}$ with $e^1_t := \gamma'(t)/\|\gamma'(t)\|$ being the tangent vector of $\gamma(t)$ of unit length. In terms of this basis, we set $\U_{ij}(t)=g(e^i_t,\U(t)e^j_t)$, $\lambda_t = 1 + \int_{0}^{t}U_{11}(s)ds$, and $f(t) = \exp(\lambda_t)$, the Riccati equation implies that
$$\frac{\mathrm{d}}{\mathrm{d}t}\U_{11}(t) + \sum_{j=1}^{n} \U_{1j}^2 =0 \text{, therefore } -\lambda'' = -\frac{\mathrm{d}}{\mathrm{d}t} \U_{11}(t) = \sum_{j=1}^{n} \U_{1j}^2 \geq \U_{11}^2 = (\lambda')^2$$
which implies further that:
$$f'' = \lambda''f+ (\lambda')^2f \leq 0.$$

We let $y_t := \mathrm{log}J_t$ where $J_t:= \mathrm{det}\A_t = \|\D\tF_t(x)\|$, $y^\psi_t:= \mathrm{log}J^\psi_t = \psi(x) -\psi \circ \tF_t(x) + y_t$ and $\alpha_t := y^\psi_t - \lambda_t$. Define
$$u(t) := e^{\alpha_t} = \frac{J^\psi_t}{f(t)},$$
since $y' = \mathrm{tr}\U$, $y'' = \mathrm{tr}(\U')$, we have
\begin{eqnarray*}
    \alpha'' = y'' - \lambda_t'' - (\psi \circ \tF_t)'' &=& -\Ric(\gamma') - \mathrm{tr}(\U^2) - \U_{11}' - \mathrm{Hess}_\psi(\gamma',\gamma') \\
    & = & -\Ric(\gamma') - \sum_{i,j=1}^{n}\U_{ij}^2 + \sum_{j=1}^n \U_{1j}^2 - \mathrm{Hess}_\psi(\gamma',\gamma')\\
    & \leq & -\Ric(\gamma') - \sum_{i,j = 2}^{n} \U_{ij}^2 - \mathrm{Hess}_\psi(\gamma',\gamma')\\
    &\leq& -\Ric(\gamma') - \frac{1}{n-1} \Big(\sum_{i=2}^{n}\U_{ii}\Big)^2 - \mathrm{Hess}_\psi(\gamma',\gamma'). \\
\end{eqnarray*}
By the Cauchy-Schwarz inequality, it follows that
\begin{eqnarray} \label{keyCSineqapp}
    \alpha'' = y'' - \lambda_t'' - (\psi \circ \tF_t)'' &\leq& -\Ric(\gamma') - \frac{(y'-\lambda')^2}{n-1} - \mathrm{Hess}_\psi(\gamma',\gamma') \nonumber \\
    &=& - \Ric_{m,N}(\gamma') - \frac{(\partial_{\gamma'}\psi)^2}{N-n} - \frac{(y'-\lambda')^2}{n-1} \nonumber \\ 
    & \leq & -\Ric_{m,N}(\gamma') - \frac{(\alpha')^2}{N-1}.
\end{eqnarray}

Inequality \eqref{keyCSineqapp} and the hypothesis $\Ric_{m,N} \geq K$ imply that $u$ satisfies the differential inequality, as in pages 143-144 of \citep{STURMKarl-Theodor2006Otgo},
\begin{eqnarray*}
    (u^\frac{1}{N-1})'' \leq -\frac{\Ric_{m,N}(\gamma')}{N-1} u^\frac{1}{N-1} \leq -\frac{Kr^2}{N-1} u^\frac{1}{N-1} \text{, where } r = d(x,\tF_1(x))
\end{eqnarray*}
which implies that $u^\frac{1}{N-1}(t)$ satisfies the concavity condition 
$$u^\frac{1}{N-1}(t) \geq \frac{\cof{K,N}((1-t)r)}{\cof{K,N}(r)}u^\frac{1}{N-1}(0) + \frac{\cof{K,N}(tr)}{\cof{K,N}(r)}u^\frac{1}{N-1}(1).$$

After multiplying $f^\frac{1}{N}$ and $u^\frac{1}{N}$ and applying the H\"older's inequality, we have
\begin{eqnarray*}
    (J^\psi_t)^\frac{1}{N} &=& (f(t)u(t))^\frac{1}{N} = f^\frac{1}{N}(t)(u^\frac{1}{N-1}(t))^\frac{N-1}{N} \\
    & \geq & ((1-t)f(0) + tf(1))^\frac{1}{N} \Big(\frac{\cof{K,N}((1-t)r)}{\cof{K,N}(r)}u^\frac{1}{N-1}(0) + \frac{\cof{K,N}(tr)}{\cof{K,N}(r)}u^\frac{1}{N-1}(1)\Big)^\frac{N-1}{N} \\
    & \geq & (1-t) \bef{1-t}{K,N}(d(x, \tF_1(x)))^{\frac{1}{N}} + t \bef{t}{K,N}(d(x, \tF_1(x)))^{\frac{1}{N}}J^\psi_1(x)^{\frac{1}{N}}
\end{eqnarray*}
showing the inequality \eqref{eqn2}.
\end{proof}

Along \citep{BacherKathrin2010OBIo}, we can show a version of Borell-Brascamp-Lieb inequality on weighted Riemannian manifolds:
\begin{theorem} \label{BBLonWRM}
    Let $f_0,f_1,h \geq 0$ be functions on a weighted Riemannian manifold $(M,g,m)$ satisfying $\Ric_{m,N} \geq Kg$. Assume that for Borel subsets $A,B \subset M$, $\int_A f_0 dm = \int_B f_1 dm = 1$ and 
    $$h(z)^{-\frac{1}{N}} \leq (1-t) \bef{1-t}{K,N}(d(x,y))^\frac{1}{N}f_0(x)^{-\frac{1}{N}} + t \bef{t}{K,N}(d(x,y))^\frac{1}{N} f_1(x)^{-\frac{1}{N}}$$ 
    for all $(x,y) \in A \times B, z \in Z_t(x,y)$, where 
    $$Z_t(x,y) := \{ z \in M | d(x,z) = td(x,y),d(z,y)=(1-t)d(x,y)\}$$
    is the set of $t$-intermediate points between $x$ and $y$. Then $\int_M h dm \geq 1$.
\end{theorem}
The proof of the above theorem is essentially the same as the proof of the main theorem of \citep{BacherKathrin2010OBIo}, which is similar to the proof of the main theorem of \citep{Cordero-ErausquinDario2001ARii} by replacing the volume distortion coefficient by the function $\bef{t}{K,N}$.
\begin{proof}
    By the assumption that $\Ric_{m,N}\geq Kg$, we have the concavity condition of the weighted Jacobian \eqref{eqn2}.

    We let $\rho_0 := f_0|_A$, $\rho_1:=f_1|_B$. There exists an optimal transference plan with density $\rho_t(x)$, satisfying Monge-Amp\`ere equation \eqref{MAE}, from $\mu_0=\rho_0 m$ to $\mu_1=\rho_1 m$ and induced by a map $\tF_t$ such that 
    $$J^\psi_t(x)\cdot \rho_t \circ \tF_t(x) = \rho_0(x) = f_0(x) \text{ for } m\text{-a.e. } x \in A, $$
    and we have, letting $y = \tF_1(x),$
    \begin{eqnarray*}
        && \Big( \frac{f_0(x)}{\rho_t(\tF_t(x))} \Big)^\frac{1}{N} \geq (1-t) \bef{1-t}{K,N}(d(x,y))^\frac{1}{N} + t\bef{t}{K,N}(d(x,y))^\frac{1}{N}\Big( \frac{f_0(x)}{f_1(\tF_1(x))} \Big)^\frac{1}{N}, \\
        &&  \Big( \frac{1}{\rho_t(\tF_t(x))} \Big)^\frac{1}{N} \geq (1-t) \bef{1-t}{K,N}(d(x,y))^\frac{1}{N} \frac{1}{f_0(x)^\frac{1}{N}} + t\bef{t}{K,N}(d(x,y))^\frac{1}{N} \frac{1}{f_1(\tF_1(x))^\frac{1}{N} } \geq h(\tF_t(x))^{-\frac{1}{N}},
    \end{eqnarray*}
    hence 
    $$h \circ \tF_t \geq \rho_t \circ \tF_t.$$

    Thus, we conclude 
    \begin{eqnarray*}
        \int_M h dm &\geq& \int_{\tF_t(A)}h dm \geq \int_M \mathbf{1}_{\tF_t(A)} \rho_t dm\\
        && = \int_A \mathbf{1}_{\tF_t(A)} \circ \tF_t(x) \cdot \rho_t \circ \tF_t(x) J^\psi_t dm = \int_A \mathbf{1}_{\tF_t(A)}(\tF_t(x)) d\mu_0 \\
        && = 1.
    \end{eqnarray*}
\end{proof}

The above theorem can be expressed in terms of the $p$-mean via the same method of scaling as the main theorem of \citep{Cordero-ErausquinDario2001ARii} was generalised to Corollary 1.1 of \citep{Cordero-ErausquinDario2001ARii}. 

\begin{corollary} \label{BBLgen}
    Let $f_0,f_1,h \geq 0$ be functions on a weighted Riemannian manifold $(M,g,m)$ satisfying $\Ric_{m,N} \geq Kg$. Assume that 
    $$ h(z) \geq \mathcal{M}^p_t \Big( \frac{f_0(x)}{\bef{1-t}{K,N}(d(x,y))}, \frac{f_1(y)}{\bef{t}{K,N}(d(x,y))} \Big) \text{ } \forall (x,y) \in M \times M, z \in Z_t(x,y) $$
    then 
    $$ \int_M h dm \geq \mathcal{M}_t^\frac{p}{1+Np} \Big( \int_M f_0 dm, \int_M f_1 dm \Big) $$
    where \begin{eqnarray*}
        \mathcal M^p_t(a,b):= \left\{ \begin{array}{cl}
        ((1-t)a^p + tb^p)^\frac{1}{p} \quad & \mbox{when $ab \neq 0, $}\\
        0 \quad & \mbox{when $ab=0$}\\
    \end{array}
    \right.
    \end{eqnarray*}
    is the $p$-mean, for $p\in \big[-\frac{1}{N},\infty\big]$, as introduced in the introduction.
\end{corollary}

\section{Rigidity of the Borell-Brascamp-Lieb inequality}
Suppose that for some $t_0 \in (0,1)$, the Borell-Brascamp-Lieb inequality on a weighted Riemannian manifold as in Theorem \ref{BBLonWRM} attains equality.

We therefore have, with the constructions above, for $\mu_0$-a.e. $x \in A$,
\begin{eqnarray*} 
        \int_M h dm &=& \int_{\tF_{t_0}(A)}h dm = \int_M \mathbf{1}_{\tF_{t_0}(A)} \rho_{t_0} dm\\
        & = & \int_A \mathbf{1}_{\tF_{t_0}(A)} \circ \tF_{t_0}(x) \cdot \rho_{t_0} \circ \tF_{t_0}(x) J^\psi_{t_0} dm = \int_A \mathbf{1}_{\tF_{t_0}(A)}(\tF_{t_0}(x)) d\mu_0 \\
        &=& 1
\end{eqnarray*}
implying that on $A$, $h\circ \tF_{t_0} = \rho_{t_0} \circ \tF_{t_0}$ almost everywhere and hence
\begin{equation} \label{conveql}
    J^\psi_{t_0}(x)^{\frac{1}{N}} = (1-t_0) \bef{1-t_0}{K,N}(d(x, \tF_1(x)))^{\frac{1}{N}} + t_0 \bef{t_0}{K,N}(d(x, \tF_1(x)))^{\frac{1}{N}}J^\psi_1(x)^{\frac{1}{N}}.
\end{equation}

By the comparison of differential inequalities, the interpolation inequality \eqref{conveql} attaining equality at one value of $t$ implies that the inequality attains equality for all $t \in [0,1]$, as follows.

\begin{lemma}
    If the weighted Jacobian $J^\psi_t$ satisfying \eqref{eqn2} attains equality at some $t_0 \in (0,1)$, then equality of the inequality \eqref{eqn2} holds for all $t \in [0,1]$.
\end{lemma}
\begin{proof}
    The result is shown by comparison of ODE. We give a proof for $K>0$ and the other situations can be handled in the same way. Recall that with the definition of $f$ and $u$ as in the proof of Lemma \ref{cvxwtj}, the inequality \eqref{eqn2} is derived by:
    \begin{eqnarray*}
    (J^\psi_t)^\frac{1}{N} &=& (fu)^\frac{1}{N} = f^\frac{1}{N}(u^\frac{1}{N-1})^\frac{N-1}{N} \\
    & \geq & ((1-t)f(0) + tf(1))^\frac{1}{N} \Big(\frac{\cof{K,N}((1-t)r)}{\cof{K,N}(r)}u^\frac{1}{N-1}(0) + \frac{\cof{K,N}(tr)}{\cof{K,N}(r)}u^\frac{1}{N-1}(1)\Big)^\frac{N-1}{N} \\
    & \geq & (1-t) \bef{1-t}{K,N}(d(x, \tF_1(x)))^{\frac{1}{N}} + t \bef{t}{K,N}(d(x, \tF_1(x)))^{\frac{1}{N}}J^\psi_1(x)^{\frac{1}{N}}.
\end{eqnarray*}

\textbf{Claim 1}: If equality holds at $t_0$, then $f(t_0) = (1-t_0)f(0) + t_0f(1)$ and $u(t_0)^\frac{1}{N-1} = \frac{\cof{K,N}((1-t_0)r)}{\cof{K,N}(r)}u^\frac{1}{N-1}(0) + \frac{\cof{K,N}(t_0r)}{\cof{K,N}(r)}u^\frac{1}{N-1}(1)$ where $r = d(x,\tF_1(x))$.

\textit{Proof of Claim 1:} Suppose, on the contrary, that the claim is not true, that $$f(t_0) > (1-t_0)f(0) + t_0f(1)$$
then we will have that $u(t_0)^\frac{1}{N-1} < \frac{\cof{K,N}((1-t_0)r)}{\cof{K,N}(r)}u^\frac{1}{N-1}(0) + \frac{\cof{K,N}(t_0r)}{\cof{K,N}(r)}u^\frac{1}{N-1}(1)$, which is a contradiction, as the differential inequality
$$(u^\frac{1}{N-1})'' \leq -\frac{Kr^2}{N-1} u^\frac{1}{N-1}$$
implies otherwise.

Or similarily, if $u(t_0)^\frac{1}{N-1} > \frac{\cof{K,N}((1-t_0)r)}{\cof{K,N}(r)}u^\frac{1}{N-1}(0) + \frac{\cof{K,N}(t_0r)}{\cof{K,N}(r)}u^\frac{1}{N-1}(1)$, we will get the same contradiction for $f$.

\textbf{Claim 2}: If $u$ satisfies $(u^\frac{1}{N-1})'' \leq -\frac{Kr^2}{N-1} u^\frac{1}{N-1}$ for $t \in (0,1)$ and there exists $t_0 \in (0,1)$ such that $u(t_0)^\frac{1}{N-1} = \frac{\cof{K,N}((1-t_0)r)}{\cof{K,N}(r)}u^\frac{1}{N-1}(0) + \frac{\cof{K,N}(t_0r)}{\cof{K,N}(r)}u^\frac{1}{N-1}(1)$, then 
$$u^\frac{1}{N-1}(t) = \frac{\cof{K,N}((1-t)r)}{\cof{K,N}(r)}u^\frac{1}{N-1}(0) + \frac{\cof{K,N}(tr)}{\cof{K,N}(r)}u^\frac{1}{N-1}(1) \quad \forall t \in [0,1].$$

\textit{Proof of Claim 2:} We define $w(t):=\frac{\cof{K,N}((1-t)r)}{\cof{K,N}(r)}u^\frac{1}{N-1}(0) + \frac{\cof{K,N}(tr)}{\cof{K,N}(r)}u^\frac{1}{N-1}(1)$. Note that $$w''(t) = -\frac{Kr^2}{N-1}w.$$
Thus, $u^\frac{1}{N-1}-w$ satisfies the differential inequality 
$$ (u^\frac{1}{N-1}-w)'' \leq \frac{Kr^2}{N-1}(w-u^\frac{1}{N-1}) \leq 0$$
and has boundary values $u^\frac{1}{N-1}(0)-w(0) = u^\frac{1}{N-1}(1)-w(1)=0$. Since also at $t_0$, $u^\frac{1}{N-1}-w=0$, by the maximum principle of subharmonic functions, we conclude $$u^\frac{1}{N-1}-w=0 \quad \forall t \in [0,1].$$

The lemma follows from Claim 2.
\end{proof}

We can thus show the main result by concluding:
\begin{theorem} \label{rigthmgen}
    Rigidity of the Borell-Brascamp-Lieb inequality on weighted Riemannian manifold $(M,g,m)$ satisfying $\Ric_{m,N} \geq Kg$ implies, for a.e. $x \in \mathrm{supp}f_0$, letting $\gamma(t):= \tF_t(x)$, the following hold:
    \begin{enumerate}
        \item For all $t \in(0,1)$, the $N$-Ricci curvature $\Ric_{m,N}(\gamma'(t))$ equals $K g(\gamma',\gamma')$, and the sectional curvatures along the geodesic $\gamma$ are identically $\frac{K}{N-1}$.
        \item The measure behaves as an $(N-n)$-th power of a linear combination of trignometric functions along the geodesic $\gamma(t)$ when $K>0$; as an $(N-n)$-degree polynomial of $t$ along $\gamma(t)$ when $K=0$; and as an $(N-n)$-th power of a linear combination of hyperbolic functions along $\gamma(t)$ when $K<0$.
    \end{enumerate}    
\end{theorem}
\begin{proof}
    We assume $N>n$ in the proof, as the case of $N=n$ reduces to \citep{BaloghZoltanM.2018EiBi}. With the above lemma and the construction of the functions $\alpha, y, \lambda_t, u, f$ and the definition of the matrix $U$ as in the proof of Lemma \ref{cvxwtj}, we have for all $t \in (0,1)$, 
    \begin{eqnarray*}
    (J^\psi_t)^\frac{1}{N} &=& (fu)^\frac{1}{N} = f^\frac{1}{N}(u^\frac{1}{N-1})^\frac{N-1}{N} \\
    & = & ((1-t)f(0) + tf(1))^\frac{1}{N} \Big(\frac{\cof{K,N}((1-t)r)}{\cof{K,N}(r)}u^\frac{1}{N-1}(0) + \frac{\cof{K,N}(tr)}{\cof{K,N}(r)}u^\frac{1}{N-1}(1)\Big)^\frac{N-1}{N} \\
    & = & (1-t) \bef{1-t}{K,N}(d(x, \tF_1(x)))^{\frac{1}{N}} + t \bef{t}{K,N}(d(x, \tF_1(x)))^{\frac{1}{N}}J^\psi_1(x)^{\frac{1}{N}}
\end{eqnarray*}
which implies that $f(t)= (1-t)f(0) + tf(1)$ and the function $u$ satisfies
    \begin{eqnarray*}
    (u^\frac{1}{N-1})'' &=& -\frac{Kr^2}{N-1} u^\frac{1}{N-1}.
    \end{eqnarray*}

    Since we have the chain of differential inequality
    $$(u^\frac{1}{N-1}) '' \leq - \frac{\Ric_{m,N}(\gamma')}{N-1} u^\frac{1}{N-1} \leq -\frac{Kr^2}{N-1} u^\frac{1}{N-1},$$
    by squeezing, it can be concluded that along the geodesic $\gamma$, $\Ric_{m,N}(\gamma')=Kg(\gamma',\gamma')$.

The differential inequality of $u$ attaining equality also implies that the differential inequality \eqref{keyCSineqapp} attains equality:
\begin{equation} \label{alphaequality}
\alpha'' = - \Ric_{m,N}(\gamma') - \frac{(\alpha')^2}{N-1}.
\end{equation}

The equality being attained has two implications. The first implication is that the matrix $V$, defined to be $\V:=(\U_{ij})_{i,j=2,...,n}$, satisfies the condition $\mathrm{tr}(\V^2) = \frac{1}{n-1} (\mathrm{tr}\V)^2$, that it equals the scalar multiple of the identity matrix $\V = \xi(t) \mathbf{\mathrm{Id}}_{n-1}$, which further implies that the sectional curvatures along the geodesic $\gamma(t)$ are all identical by the Riccati equation.

The second implication is that the Cauchy-Schwarz inequality between the tuples $$\Big(-\frac{\partial_{\gamma'}\psi}{\sqrt{N-n}},\frac{y'-\lambda'}{\sqrt{n-1}}\Big) \text{ and }(\sqrt{N-n},\sqrt{n-1})$$ attains equality, which holds only if
$$\frac{-\partial_{\gamma'}\psi}{y'-\lambda'} = \frac{N-n}{n-1}.$$

Thus, having $y'-\lambda' = \mathrm{tr} \V = (n-1) \xi(t)$, we have $-\partial_{\gamma'} \psi = (N-n) \xi(t)$, and $\alpha' = y'-\lambda'-\partial_{\gamma'}\psi = (N-1) \xi(t)$. Hence equation \eqref{alphaequality} gives, when $\Ric_{m,N} = Kg$,
$$\xi'(t) + \xi^2(t) = -\frac{Kr^2}{N-1}\text{, where } r = d(x,\tF_1(x))$$
and the result on the sectional curvatures thus follows from the Riccati equation.

To derive information on the behaviour of measure along $\gamma(t)$, we recall the expression of the $N$-Ricci curvature $$\Ric_{m,N} = \Ric(\gamma') + \mathrm{Hess}_{\psi}(\gamma',\gamma') - \frac{(\partial_{\gamma'}\psi)^2}{N-n}.$$
Given $\Ric_{m,N}(\gamma') = Kr^2$ and $\Ric(\gamma') = \frac{n-1}{N-1}Kr^2$, we can derive the following second order differential equation of $\psi$:
$$\frac{\mathrm{d}^2}{\mathrm{d}t^2}\psi \circ \gamma (t) - \frac{1}{N-n} \Big[\frac{\mathrm{d}}{\mathrm{d}t}(\psi \circ \gamma (t))\Big]^2 = \frac{N-n}{N-1}Kr^2.$$

Therefore, it can be deduced that $\e^{-\psi \circ \gamma(t)}$ satisfies the differential equation
$$\frac{\mathrm{d}^2}{\mathrm{d} t^2} \e^\frac{-\psi \circ \gamma(t))}{N-n} = -\frac{Kr^2}{N-1} \e^\frac{-\psi \circ \gamma(t)}{N-n}.$$
The solution of the differential equation yields
\begin{eqnarray*}
    \e^{-\psi \circ \tF_t(x)} = \left\{ \begin{array}{cl}
        \Big[ C_0\sin\Big( \sqrt{\frac{K}{N-1}}rt\Big) + C_1 \cos \Big(\sqrt{\frac{K}{N-1}}rt\Big) \Big]^{N-n} \quad & \mbox{when $K>0$,}\\
        (C_0t+C_1)^{N-n} \quad & \mbox{when $K=0$,}\\
        \Big[ C_0 \sinh\Big(\sqrt{\frac{-K}{N-1}}rt\Big) + C_1 \cosh\Big(\sqrt{\frac{-K}{N-1}}rt\Big) \Big]^{N-n} \quad & \mbox{when $K<0$}
    \end{array}
    \right.
\end{eqnarray*}
where $C_0$ and $C_1$ are constants determined by the initial values $\e^{-\psi(x)}$ and $\partial_{\gamma'(0)}\e^{-\psi(x)}$
\begin{eqnarray*}
    C_0 = \frac{-\partial_{\gamma'(0)}\psi(x)}{N-n}\e^\frac{-\psi(x)}{N-n}\text{, }C_1 = \e^\frac{-\psi(x)}{N-n}.
\end{eqnarray*}
\end{proof}

\begin{remark} 
When the dimension $n$ of $M$ is less than $N$, the density function $e^{-\psi \circ \tF_t(x)}$ takes the value of $0$ at $t=-\frac{1}{r}\sqrt{\frac{N-1}{K}}\arctan\big( \frac{C_1}{C_0} \big)$ when $K>0$; at $t=-\frac{C_1}{C_0}$ when $K=0$; or at $t = -\frac{1}{r} \sqrt{\frac{N-1}{-K}}\mathrm{arctanh}\big(\frac{C_1}{C_0}\big)$ when $K<0$ and $|C_1|<|C_0|$.

Nonetheless, this does not lead to a contradiction. On a smooth weighted Riemannian manifold $M$, the initial value $e^{-\psi(x)}$ is nonzero and the point at which the density $e^{-\psi}$ vanishes is not on the geodesic $\tF_t(x)$ for all $t \in [0,1]$ for a.e. $x \in \mathrm{supp}f_0$, hence outside the support of $f_0,f_1$ and $h$ and can be resolved by mollification. Therefore it is possible to attain equality of the Borell-Brascamp-Lieb inequality on smooth weighted Riemannian manifolds. (See also Remark \ref{BMRmk}.)
\end{remark}

For completeness, we include also a rigidity result on unweighted Riemannian manifolds, which is a direct consequence of Theorem \ref{rigthmgen} and can be considered as an addition of the consideration of dimension bounds to the curvature rigidity result in Theorem 4.1 of \citep{BaloghZoltanM.2018EiBi}.

\begin{corollary}
    On a Riemannian manifold $(M,g,\mathrm{vol}_g)$ satisfying $\Ric_{N}\geq Kg$, rigidity of the Borell-Brascamp-Lieb inequality implies that the dimension $n$ of $M$ coincides with $N$ and the sectional curvatures along the geodesic $\gamma(t)$ are identically $\frac{K}{N-1}$. 
\end{corollary}

\section{The Brunn-Minkowski Inequality and its Rigidity}
The Brunn-Minkowski inequality on Riemannian manifolds with $\Ric_{m,N} \geq Kg$, established in \citep{STURMKarl-Theodor2006Otgo} is closely related to the Borell-Brascamp-Lieb inequality, whose rigidity can also be considered as a consequence of the rigidity of the Borell-Brascamp-Lieb inequality.

For Borel subsets $A,B \subset M$ of positive and finite measure, the Brunn-Minkowski inequality is a consequence of the Borell-Brascamp-Lieb inequality, by an appropriate selection of functions $f_0,f_1,h$ in the Borell-Brascamp-Lieb inequality (in the form of Corollary \ref{BBLgen}). 

We pick $f_0:=\bef{1-t}{K,N}(\Theta_{A,B})\mathbf{1}_A$, $f_1:= \bef{t}{K,N}(\Theta_{A,B}) \mathbf{1}_B$ and $h:= \mathbf{1}_{Z_t(A,B)}$ where $\Theta_{A,B} := \inf_{(x,y) \in A\times B} d(x,y)$ when $K \geq 0$ and $\Theta_{A,B} := \sup_{(x,y) \in A\times B} d(x,y)$ when $K < 0$. It is easy to see that the condition $ h(z) \geq \mathcal{M}^p_t \Big( \frac{f_0(x)}{\bef{1-t}{K,N}(d(x,y))}, \frac{f_1(y)}{\bef{t}{K,N}(d(x,y))} \Big)$ for $z \in Z_t(x,y)$ is satisfied. Letting $p = + \infty$, the Borell-Brascamp-Lieb inequality implies the Brunn-Minkowski inequality:
\begin{multline*}
 m(Z_t(A,B)) = \int_M h dm \geq \mathcal{M}_t^\frac{1}{N} \Big( \int_M f_0 dm, \int_M f_1 dm \Big) \\
  = \Big( (1-t)\bef{1-t}{K,N}(\Theta_{A,B})^\frac{1}{N}m(A)^\frac{1}{N} + t \bef{t}{K,N}(\Theta_{A,B})^\frac{1}{N}m(B)^\frac{1}{N} \Big)^N.
\end{multline*}

\begin{remark}
    The fact that on weighted Riemannian manifolds satisfying $\Ric_{m,N}\geq Kg$ the Borell-Brascamp-Lieb inequality implies Brunn-Minkowski inequality shows the equivalence between Borell-Brascamp-Lieb inequality and $\Ric_{m,N}\geq Kg$. `$\Ric_{m,N}\geq Kg \Rightarrow$ BBL' is known, and for the inverse implication, we have that the Borell-Brascamp-Lieb inequality implies Brunn-Minkowski inequality, which by \citep{MagnaboscoMattia2024TBii} implies the $\mathrm{CD}(K,N)$ condition, equivalent to $\Ric_{m,N} \geq Kg$.
\end{remark}

\subsection{Rigidity in the null curvature bound}
In the case of null-curvature bound, the Brunn-Minkowski inequality takes the form 
\begin{equation} \label{BMFlat}
m(Z_t(A,B))^\frac{1}{N} \geq (1-t) m(A)^\frac{1}{N} + tm(B)^\frac{1}{N}.
\end{equation}

Assuming that the sets $A,B$ differ from convex sets up to a set of measure $0$ and that $0<m(A),m(B)< \infty$, the result below is a rather straightforward consequence of the rigidity of the Borell-Brascamp-Lieb inequality in $\CD(0,N)$ manifolds. 

\begin{theorem}
    Rigidity of the Brunn-Minkowski inequality on weighted Riemannian manifolds such that $\Ric_{m,N} \geq 0$ implies that for a.e. $x \in A$, the sectional curvatures along the geodesic of transport from $A$ to $B$ starting at $x$ are identically $0$, and the weight behaves as a polynomial of degree $N-n$ along the paths of transportation.
\end{theorem}
\begin{proof}
    Equality in the Brunn-Minkowski inequality reads
    $$m(Z_t(A,B))^\frac{1}{N} = (1-t)m(A)^\frac{1}{N} + tm(B)^\frac{1}{N}.$$    
    Picking the functions to be the uniform distribution functions with supports being the respective sets, the result follows from the rigidity of the Borell-Brascamp-Lieb inequality, that for $m$-a.e. $x \in A$, along the geodesic $t \mapsto \tF_t(x)$, the sectional curvatures are constantly $0$ and the measure along it behaves as
    $$\e^{-\psi \circ \tF_t(x)} = (C_0t+C_1)^{N-n}$$
    as in Theorem \ref{rigthmgen}.
\end{proof}
\begin{remark} \label{BMRmk}
    When $A$ is a point $\{x\}$ and $B$ is a ball centred at $x$, rigidity of the Brunn-Minkowski inequality would imply equality in the Bishop-Gromov volume comparison, which further implies that $B$ is a part of a flat cone due to the main result of \citep{DePhilippisGuido2016Fvct}, with a singularity at $\{x\}$ if $M$ is complete, unless the dimension $n$ of $M$ is $N$. Whilst it may be difficult to construct nontrivial examples of sets that attain equality of the Brunn-Minkowski inequality on complete smooth Riemannian manifolds with dimension less than $N$, it is possible to attain equality in the Brunn-Minkowski inequality on geodesically convex smooth open manifolds. 
\end{remark}

\begin{examp}
    Let $M$ be the Euclidean half space $\reals^n_+:=\{x=(x^1,...,x^n) \in \reals^n | x^n>0\}$ endowed with the measure $m = \|x\|^{N-n} \mathcal L^n$. Let $A=B_r((0,...,1))$ be a ball of radius $r<1$ centred at $(0,...,1)$ and $B = 2A = B_{2r}((0,...,2))$. We have $Z_t(A,B) = (1+t)A$, and 
    \begin{eqnarray*}
        m(Z_t(A,B)) &=& \int_{Z_t(A,B)} dm = \int_{(1+t)A} \|x\|^{N-n} d\mathcal L^n(x)= \int_A (1+t)^n \|(1+t)x\|^{N-n} d\mathcal L^n(x) \\
        &=& (1+t)^N m(A) = \Big( (1-t)m(A)^{\frac{1}{N}} + tm(B)^{\frac{1}{N}} \Big)^N.
    \end{eqnarray*}
\end{examp}

\subsection{Rigidity in the positive and negative curvature bounds}
The Brunn-Minkowski inequality takes the form 
$$m(Z_t(A,B))^\frac{1}{N} \geq (1-t)\bef{1-t}{K,N}\Big(\inf_{(x,y) \in A\times B} d(x,y)\Big)^\frac{1}{N}m(A)^\frac{1}{N} + t \bef{t}{K,N}\Big(\inf_{(x,y) \in A\times B} d(x,y)\Big)^\frac{1}{N}m(B)^\frac{1}{N}$$
when $K>0$, and
$$m(Z_t(A,B))^\frac{1}{N} \geq (1-t)\bef{1-t}{K,N}\Big(\sup_{(x,y) \in A\times B} d(x,y)\Big)^\frac{1}{N}m(A)^\frac{1}{N} + t \bef{t}{K,N}\Big(\sup_{(x,y) \in A\times B} d(x,y)\Big)^\frac{1}{N}m(B)^\frac{1}{N}$$
when $K<0$.

By the scaling property of the $p$-mean, similar to the one used in the Theorem 4.1(ii) and Theorem 4.2 of \citep{BaloghZoltanM.2018EiBi}, we have:
\begin{equation} \label{befeq}
    1 = \frac{\bef{1-t}{K,N}(\Theta_{A,B})}{\bef{1-t}{K,N}(d(x,\tF_1(x)))} = \frac{\bef{t}{K,N}(\Theta_{A,B})}{\bef{t}{K,N}(d(x,\tF_1(x)))}
\end{equation}
and due to the monotonicity of the function $\bef{t}{K,N}(r)$ in $r$, $d(x,\tF_1(x)) = \Theta_{A,B}$ for a.e. $ x \in A$. Thus, by the same proof of Theorem 4.2 of \citep{BaloghZoltanM.2018EiBi}, we conclude the following result on rigidity.

\begin{theorem}
    On a weighted Riemannian manifold with $\Ric_{m,N} \geq Kg$ where $K>0$, if compact convex sets $A,B \subset M$ with finite and positive measure satisfies equality of the Brunn-Minkowski inequality for some $t_0 \in (0,1)$, then $A = B = Z_t(A,B)$ for all $t\in [0,1]$ up to a set of null measure. Equality of the Brunn-Minkowski inequality cannot hold for compact convex sets of finite and positive measure on weighted Riemannian manifolds satisfying $\Ric_{m,N} \geq Kg$ when $K<0$.
\end{theorem}
\begin{proof}
    For the case when $K>0$, we separate the situation when $A \cap B = \emptyset$ and when $A \cap B \neq \emptyset$.

    When $A \cap B \neq \emptyset$, $\Theta_{A,B} = 0$ and by equation \eqref{befeq}, $\tF_1(x) = x$ for almost all $x \in A$, implying that $A=B$ up to a set of measure $0$, which further implies that $A = B = Z_t(A,B)$ up to a set of null measure.

    When $A$ and $B$ are disjoint, we can apply the same argument constructed in \citep{BaloghZoltanM.2018EiBi} with $\frac{\cof{K,N}(tx)}{\cof{K,N}(x)}$ replaced by $\bef{t}{K,N}(x)$. We set $s:= \Theta_{A,B} = \inf_{(x,y) \in A\times B} d(x,y) = d(x,\tF_1(x))$ for almost every $x \in A$. Fix $x \in A$ such that $d(x,\tF_1(x))=s$, and $B_s(\tF_1(x)) \subset \bigcup_{y \in B} B_s(y)$. By construction, $B_s(\tF_1(x))$ is disjoint with $A$. Fix $s_0 \in (0,s)$ and for every $0 < s' < s_0$, for $z_{s'} \in Z_{\frac{s'}{2s}}(x,\tF_1(x))$, $B_\frac{s'}{2}(z_{s'})$ is disjoint with $A$. This implies that $A$ is porous at $x$, that for every $s'$ sufficiently small, there is a point $z_{s'}$ of distance $s'$ to $x$, such that there is an open ball with radius $\frac{s'}{2}$ that is outside $A$ centred at which. In this case $m(A)$ cannot be positive, which is a contradiction to the assumption that $m(A)>0$.

    In the case when $K<0$, by a similar construction of contradiction as in the case when $K>0$ and $A,B$ are disjoint, it can also be shown that for points $x \in A$ such that $d(x,\tF_1(x)) = \Theta_{A,B}$, $A$ is porous at $x$, implying that $m(A)=0$ which contradicts the assumption of $A$ having positive measure.
\end{proof}

\section{Further Problems}
Theorem \ref{rigthmgen} is concerned with the simplest situation of the rigidity of the Borell-Brascamp-Lieb inequality in metric measure geometry. There are a number of directions that it can be further generalised into.
\begin{enumerate}
    \item After the result on rigidity of the Borell-Brascamp-Lieb inequality, it is natural to consider the quantitative stability of it on smooth weighted Riemannian manifolds with $\Ric_{m,N}\geq Kg$ for $N \in [n,\infty)$.
    \item Another natural further problem is to consider the rigidity of the Borell-Brascamp-Lieb inequality with weaker dimension bounds, namely on smooth Riemannian manifolds with $\Ric_{m,N}\geq Kg$ for $N = \infty$ or $N \leq 0$.
    \item It is possible to generalise Theorem \ref{rigthmgen} to Finsler manifolds, on which the curvature-dimension condition is described in \citep{OhtaShin-ichi2009Fii}, by considering the equality of the Finsler interpolation inequality presented in the same paper.
    \item Following the proof of the Borell-Brascamp-Lieb inequality on metric measure spaces in \citep{BacherKathrin2010OBIo}, it is also natural to consider the rigidity of this inequality in a synthetic way, on $\CD(K,N)$ or $\mathrm{RCD}(K,N)$ metric measure spaces.
    \item Considering the Borell-Brascamp-Lieb inequality together with its rigidity and stability on Lorentzian manifolds with timelike Ricci curvature bounds is also a direction of generalisation, under the framework of optimal transport on Lorentzian manifolds constructed in \citep{McCannRobertJ.2020DcoB} and \citep{MondinoAndrea2023Aotf}. Recently, in \citep{CavallettiFabio2025Asii}, Cavalletti-Mondino proved the Brunn-Minkowski inequality on Lorentzian manifolds with timelike Ricci curvature bound, and in \citep{FarooquiOsama2025Ebtt} Farooqui showed the equivalence between timelike Brunn-Minkowski inequality and timelike Ricci curvature bound.
\end{enumerate}

\bibliographystyle{plainnat}
\bibliography{Primo_BibTeX.bib}

\end{document}